\documentclass[a4paper,11pt,makeidx]{amsart}
\usepackage{a4wide}\parskip=3pt 
\usepackage{todonotes}
\usepackage{amscd}
\usepackage{xypic}  %commu 
\usepackage{amssymb}
\usepackage{amsthm}
\usepackage{epsf}
\usepackage{mathtools}
\usepackage{hyperref}
\mathtoolsset{showonlyrefs}
\makeindex
\newtheoremstyle{break}% name
  {9pt}%      Space above, empty = `usual value'
  {9pt}%      Space below
  {\itshape}% Body font
  {}%         Indent amount (empty = no indent, \parindent = para indent)
  {\bfseries}% Thm head font
  {.}%        Punctuation after thm head
  {\newline}% Space after thm head: \newline = linebreak
  {}%         Thm head spec
\theoremstyle{break}
\newtheorem{bthm}{Theorem}

\newtheorem{bcor}{Corollary}
\theoremstyle{plain}
\newtheorem{thm}{Theorem}[section]
\newtheorem{cor}[thm]{Corollary}

\newtheorem{lemma}[thm]{Lemma}

\newtheorem{prop}[thm]{Proposition}
\newtheorem{defn}[thm]{Definition}

\newcommand{\qee}{\mbox{\hspace{0.2mm}}\hfill$\triangle$}
\theoremstyle{remark}
\newtheorem{rem}[thm]{Remark}
\newtheorem{ex}[thm]{Example}
\newenvironment{remark}{\begin{rem}\rm}{\qee\end{rem}}
\newenvironment{example}{\begin{ex}\rm}{\qee\end{ex}}

\DeclareMathOperator{\NL}{NL}

\def\Cl{\operatorname{Cl}}

\def\Sing{\operatorname{Sing}}

\def\codim{\operatorname{codim}}

\def\Im{\operatorname{Im}}
\def\rk{\operatorname{rk}}

\def\max{\operatorname{max}}

\def\CC{{\mathbb C}}
\def\ZZ{{\mathbb Z}}
\def\NN{{\mathbb N}}

\def\QQ{{\mathbb Q}}
\def\PP{{\mathbb P}}

\def\N{{\mathbb N}}
\def\O{{\mathcal O}}
\def\I{{\mathcal J}}

\def\F{{\mathcal F}}

\def\J{{\mathfrak J}}
\def\*{\otimes}
\def\+{\oplus}                   % direct sum
\def\*{\otimes}                  % tensor product
\def\Ext{\operatorname{Ext}}
\def\Hom{\operatorname{Hom}}

\def\Pic{\operatorname{Pic}}

\def\det{\operatorname{det}}

\begin{document}

\title[General components of the Noether-Lefschetz locus on normal threefolds]{Existence and density of general components of \\[5pt] the Noether-Lefschetz locus on normal threefolds}

\author[U. Bruzzo, A. Grassi, and A.F. Lopez]{Ugo Bruzzo*, Antonella Grassi** and Angelo Felice Lopez***}

\thanks{* Research partially supported by  PRIN ``Geometria delle variet\`a algebriche'' and GNSAGA-INdAM}

\thanks{* and **  Research partially supported by  the University of Pennsylvania Department of Mathematics Visitors Fund}

\thanks{*** Research partially supported by the MIUR national project ``Moduli spaces and Their Applications" FIRB 2012}

\address{\hskip -.43cm SISSA (Scuola Internazionale Superiore di Studi Avanzati), Via Bonomea 265, 34136 Trieste, Italia;
IGAP (Institute for Geometry and Physics), Trieste;
INFN (Istituto Nazionale di Fisica Nucleare), Sezione di Trieste; Arnold-Regge Institute for Algebra, Geometry and Theoretical Physics, Torino. E-mail {\tt bruzzo@sissa.it}}

\address{\hskip -.43cm Department of Mathematics, University of Pennsylvania, David Rittenhouse Laboratory, 209 S 33rd Street, Philadelphia, PA 19104, USA. e-mail {\tt grassi@math.upenn.edu}}

\address{\hskip -.43cm Dipartimento di Matematica e Fisica, Universit\`a di Roma
Tre, Largo San Leonardo Murialdo 1, 00146, Roma, Italy. e-mail {\tt lopez@mat.uniroma3.it}}


\begin{abstract} We consider the Noether-Lefschetz problem for surfaces in $\QQ$-factorial normal 3-folds with rational singularities. We show the existence of components of the Noether-Lefschetz locus of maximal codimension,
and that there are indeed infinitely many of them. Moreover, we show that their union is dense  in the natural topology. \end{abstract}

\maketitle

\section{Introduction}
\label{intro}

Let $Y$ be a smooth complex variety and let $D$ be a smooth ample divisor. Among several classical results in this setting, stand for importance the Noether-Lefschetz type results, namely that the natural restriction map $i_D : \Pic(Y) \to \Pic(D)$ is an isomorphism if $\dim Y \geq 4$, and, in many cases, if $\dim Y = 3$ and $D$ is very general in its linear system.

In the latter case, the locus of smooth surfaces $D$ such that $i_{D}$ is not surjective, is called the Noether-Lefschetz locus of $|D|$. This gives rise to countably many subvarieties of $|D|$, called components of the Noether-Lefschetz locus. The study of the geometry of such components is nowadays itself a classical subject (see, to mention a few, \cite{cggh,Green1,Green3,Voisin88Precision,VoisinComposantesDeNL,CilibertoHarriaMirandaNLLocus,Lopez,cl,OtwinowskaJAG,otw2,LopMac})
%\todo{Ho messo insieme tutte queste referenze; va bene?}
and is basically divided in two parts: the study of low or high, in fact maximal, codimension components.

In the present paper we consider components of maximal codimension, the main goal being to study their existence, the fact that there are infinitely many such components and that they are dense in the natural topology. Moreover, we work on an ambient threefold with mild singularities. To our knowledge this is a novelty, if we exclude \cite{CoxNoether} and the toric case \cite{BG1, bg2}, from which this work drew inspiration. 
%\todo{Qui sono andato a capo. Mi sembra esteticamente meglio. Inoltre, nelle note sotto, ho precisato ***}
%We now explain our set up.
%\begin{AL}(NUOVA)
%Ho percentualizzato tutte le mie note vecchie (versione precedente) in quanto penso siamo d'accordo. Le nuove note le ho denotate con AL(NUOVA).
%\end{AL}

Let $X$ be a complex normal irreducible threefold with rational singularities (we shall always consider varieties over the complex numbers), and let $L$ be a very ample line bundle on $X$. Given a normal surface $S \in |L|$ it follows, by Mumford's vanishing \cite[Thm.\ 2]{mu}, that $H^1(S, -mL_{|S})=0$ for every $m \ge 1$, whence, the restriction map
\[i_S : \Pic(X) \to \Pic(S) \]
%\todo{Ho aggiunto "$i_S$", coerentemente con la mappa $i_D$ definita prima} 
is injective by \cite[Expos\'e XII, Cor.\ 3.6]{SGA2-XI}. 
%\begin{AG}Aggiunto:
%\end{AG} 
%\begin{AL}(NUOVA)
%OK
%\end{AL}

Recall that for a normal variety $Y$ we define $\rho (Y)$ to be the rank of $\Pic (Y) \otimes \mathbb Q$. We can therefore define (in analogy with the smooth case):
\begin{defn}
Let $X$ be a normal irreducible threefold with rational singularities, and let $L$ be a very ample line bundle on $X$. Let $U(L)$ be the open subset of $|L|$ parametrizing irreducible normal surfaces with rational singularities. 

The {\bf Noether-Lefschetz locus} of $(X,L)$ is
\[\NL(L) = \{ S \in U(L) : \rho(S) > \rho(X) \}. \] 
If, for a very general $S \in |L|$, we have that $\rho(S) = \rho(X)$, then $\NL(L)$ is a countable union of proper subvarieties of $U(L)$, which we call {\bf components of the Noether-Lefschetz locus}. 
\end{defn}
%\todo{Mettere un remark anticipando i riferimenti a Seideberg etc.}
As in the case of $\PP^3$, assuming that $\omega_X(L)$ is globally generated and $h^2(\O_X) = h^3(\O_X)$, it is not difficult to see (Proposition \ref{codmax}) that the components of the
%\todo{Ho aggiunto "the"}
Noether-Lefschetz locus $\NL(L)$ exist and have a maximal possible codimension $h^0(\omega_X(L))$ in $U(L)$. 

Our first result is that, in many cases, we can get the same results as for $\PP^3$, namely that components of maximal codimension exist:
%\begin{AL}
%Rispetto alla versione precedente, ho tolto la (v) e la condizione $H^2(H) = 0$ in (ii), che non servono pi\`u.
%\end{AL}

\begin{bthm}
\label{esistenza}
Let $X$ be a normal, $\QQ$-factorial, irreducible threefold with rational singularities, and let $H$ be a very ample line bundle on $X$. Suppose that
\begin{itemize}
\item[(i)] $H^i(\O_X) = 0$ for $i > 0$;
\item[(ii)] $H^1(H) = 0$;
\item[(iii)] $H^0(\omega_X(H)) = 0$.
\end{itemize}
Let $d \ge 2$ be an integer such that 
\begin{itemize}
\item[(iv)] $\omega_X(dH)$ is globally generated.
\end{itemize}
%\todo{Uniformare la nota\-zione dei luoghi $W$}
Then there is a component $W(dH)$ of the Noether-Lefschetz locus $\NL(dH)$ such that
%\todo{Ho aggiunto "locus"}
\[ \codim_{U(dH)} W(dH) = h^0(\omega_X(dH)). \]
\end{bthm}

%VERSIONE PRECEDENTE
%\begin{bthm}
%\label{esistenza}
%Let $X$ be a normal, $\QQ$-factorial, irreducible threefold with rational singularities, and let $H$ be a very ample line bundle on $X$. Suppose that
%\begin{itemize}
%\item[(i)] $H^i(\O_X) = 0$ for $i > 0$;
%\item[(ii)] $H^1(H) = H^2(H) = 0$;
%\item[(iii)] $H^0(\omega_X(H)) = 0$.
%\end{itemize}
%Let $d \ge 2$ be an integer such that 
%\begin{itemize}
%\item[(iv)] $\omega_X(dH)$ is globally generated;
%\item[(v)] $H^1(dH) = 0$.
%\end{itemize}
%%\todo{Uniformare la nota\-zione dei luoghi $W$}
%Then there is a component $W(dH)$ of the Noether-Lefschetz locus $\NL(dH)$ such that
%\todo{Ho aggiunto "locus"}
%\[ \codim_{U(dH)} W(dH) = h^0(\omega_X(dH)). \]
%\end{bthm}

%VECCHIA FRASE
%Moreover, such components are dense in the natural topology:
Moreover, this gives density in the natural topology:
%VECCHIO STATEMENT
%\begin{bcor}
%\label{densita'}
%Let $X$ be a normal, $\QQ$-factorial, irreducible threefold with rational singularities, let $H$ be a very ample line bundle on $X$ and let $d \ge 2$ be an integer such that (i)-(v) of Theorem \ref{esistenza} are satisfied. Then the components of maximum codimension $h^0(\omega_X(dH))$ are dense, in the natural topology, in the Noether-Lefschetz locus $\NL(dH)$.
%\end{bcor}
%\begin{AL}
%Rispetto alla versione precedente, nel Cor.1 ora sono da soddisfare solo (i)-(iv) del Thm.1.
%\end{AL}

\begin{bcor}
\label{densita'}
Let $X$ be a normal, $\QQ$-factorial, irreducible threefold with rational singularities, let $H$ be a very ample line bundle on $X$ and let $d \ge 2$ be an integer such that (i)-(iv) of Theorem \ref{esistenza} are satisfied. Then the Noether-Lefschetz locus $\NL(dH)$ is dense, in the natural topology, in $U(dH)$.
\end{bcor}
%VERSIONE PRECEDENTE
%\begin{bcor}
%\label{densita'}
%Let $X$ be a normal, $\QQ$-factorial, irreducible threefold with rational singularities, let $H$ be a very ample line bundle on $X$ and let $d \ge 2$ be an integer such that (i)-(v) of Theorem \ref{esistenza} are satisfied. Then the Noether-Lefschetz locus $\NL(dH)$ is dense, in the natural topology, in $U(dH)$ .
%\end{bcor}
%We observe that Theorem \ref{esistenza} and Corollary \ref{densita'} apply, for example, on many Fano threefolds, see Remark \ref{fano}. As a matter of fact these results are a sort of sample, as by applying our main tools, Lemma \ref{lemmetto} and Corollary \ref{riso}, one can choose a suitable threefold and suitable vector bundles, and get similar results.
In the special case of toric threefolds, we obtain:
%\todo{Abbiamo shiftato $d\to d-2$ e messo $d\ge 2$ perch\'e sembra sufficiente}
%\todo{Sono ritornato alla versione precedente, ma con $d\ge 0$. Mi sembra meglio cosi'}
%\begin{AL}
%Rispetto alla versione precedente ho tolto, sia in Thm.2 che in Cor.2, la condizione "$- K_{\PP_{\Sigma}} + dH$ is very ample" che segue dall'ipotesi $-K_{\PP_{\Sigma}} - 2 H$ nef.
%\end{AL}
\begin{bthm}
\label{esistenza-torica} 
Let $\PP_{\Sigma}$ be a projective simplicial Gorenstein toric threefold and let $H$ be a very ample line bundle on $X$ such that $-K_{\PP_{\Sigma}} - 2 H$ is nef. Then, for every $d \ge 0$, there is a component $W(d)$ of the Noether-Lefschetz locus $\NL(- K_{\PP_{\Sigma}} + dH)$ such that
\[ \codim W(d) = h^0(dH). \]
\end{bthm}
%VERSIONE PRECEDENTE
%\begin{bthm}
%\label{esistenza-torica} 
%Let $\PP_{\Sigma}$ be a projective simplicial Gorenstein toric threefold and let $H$ be a very ample line bundle on $X$ such %that $-K_{\PP_{\Sigma}} - 2 H$ is nef. Then, for every $d \ge 0$ such that $- K_{\PP_{\Sigma}} + dH$ is very ample, there is a %component $W(d)$ of the Noether-Lefschetz locus $\NL(- K_{\PP_{\Sigma}} + dH)$ such that
%\[ \codim W(d) = h^0(dH). \]
%\end{bthm}
Note that  the hypotheses 
%$\PP_{\Sigma}$ Gorenstein, $H$ ample and $-K_{\PP_{\Sigma}} - 2 H$  nef  
 in  the above theorem imply that $ \PP_{\Sigma}$  is a Fano threefold.
%\noindent and
%\todo{Ho tolto "and" ed aggiunto "Moreover"}
Moreover, combining with \cite{bg2}:
%VECCHIO STATEMENT
%\begin{bcor}
%\label{densita'-torica}
%Let $\PP_{\Sigma}$ be a projective simplicial Gorenstein toric threefold and let $H$ be a very ample line bundle on $X$ such that $-K_{\PP_{\Sigma}} - 2 H$ is nef. Then, for every integer $d \ge 0$
%\todo{Anche qui $d \ge 0$} such that $- K_{\PP_{\Sigma}} + dH$ is very ample, the components of maximum codimension $h^0(dH)$ are dense, in the natural topology, in the Noether-Lefschetz locus $\NL(- K_{\PP_{\Sigma}} + d H)$.
%\end{bcor}
%\begin{AG} Ho aggiunto i limiti sulla codimensione
%\end{AG}
%\begin{AL}(NUOVA)
%OK
%\end{AL}

\begin{bcor}
\label{densita'-torica}
Let $\PP_{\Sigma}$ be a projective simplicial Gorenstein toric threefold and let $H$ be a very ample line bundle on $X$ such that $-K_{\PP_{\Sigma}} - 2 H$ is nef. Then, for every integer $d \ge 0$,
%\todo{Anche qui $d \ge 0$} 
the Noether-Lefschetz locus $\NL(- K_{\PP_{\Sigma}} + d H)$ is dense, in the natural topology, in $U(- K_{\PP_{\Sigma}} + d H)$.

If   $-K_{\PP_{\Sigma}} \neq 2 H$ 
%\begin{AL}(NUOVA)
%Penso che sia $\-K_{\PP_{\Sigma}} \neq - 2 H$.
%\end{AL} 
and $d\geq 3$ then $ d \leq \codim \NL(- K_{\PP_{\Sigma}} + d H) \leq h^0(dH).$
\end{bcor}
%\begin{AL}(NUOVA)
%Per applicare Cor.4.13 in [7] non serve H primitivo?
%\end{AL}
%\begin{AG}(Nuova) Non mi pare. Con Ugo avevamo usato primitivo per gli esempi ed i conti con le Fano, ma non nelle dimostrazioni.
%\end{AG}
%\begin{AL}(NUOVA2)
%Sono d'accordo. In effetti anche se H non \`e primitivo basta prendere $H=s\eta$ con $\eta$ primitivo e tutto %funziona (anzi, se fosse $s\ge 2$ il bound sarebbe migliore....)
%\end{AL}
%VERSIONE PRECEDENTE
%\begin{bcor}
%\label{densita'-torica}
%Let $\PP_{\Sigma}$ be a projective simplicial Gorenstein toric threefold and let $H$ be a very ample line bundle on $X$ such that $-K_{\PP_{\Sigma}} - 2 H$ is nef. Then, for every integer $d \ge 0$
%\todo{Anche qui $d \ge 0$} 
%such that $- K_{\PP_{\Sigma}} + dH$ is very ample, the Noether-Lefschetz locus $\NL(- K_{\PP_{\Sigma}} + d H)$ is dense, in the natural topology, in $U(- K_{\PP_{\Sigma}} + d H)$.
%\end{bcor}
%\vskip .5cm
%
%\noindent {\bf @@@@@@@@@@@@@@@@@@@@@@@
%
%\noindent NOTA BENE}
%
%1) Nel teorema \ref{esistenza-torica} ho preferito cambiare notazione, per non creare confusione con i vostri risultati, dove il luogo di Noether-Lefeschez \`e un po' differente... 
%
%2) Inoltre serve che $\PP_{\Sigma}$ sia Gorenstein, dato che serve che $- K_{\PP_{\Sigma}} + dH$ sia Cartier. Scusate, di questo mi accorgo solo ora!
%\begin{AG} Modifiche da qui alla fine dell'introduzione.
%\end{AG}
%\begin{AL}(NUOVA)
%OK
%\end{AL}

It can be easily verified that several families of varieties satisfy the hypotheses of the above Theorems and Corollaries. 
We  present some examples in Section \ref{esempi}; we also discuss the relation with Castelnuovo-Mumford regularity.

As this paper was being completed, we received a preprint from O. Benoist \cite{Benoist} that contains an application of density results for Noether-Lefschetz loci in the  context of  studying properties of real polynomials which are a sum of squares, related to ``Hilbert's 17th problem". Even though both papers obtain density results by using determinantal curves, there are substantial differences in both the results and the methods. Benoist's paper, as well as \cite{MP}  and \cite{SOS} use the density results for  Noether-Lefschetz loci in smooth %loci
ambient varieties. The current  paper opens the way to study such problems in a more general context.

%Even though his density results are obtained, similarly to ours, by using determinantal curves, there are substantial differences between the two papers, both in the goals, and in the methods.
% {Theorem \ref{esistenza}}.
%\todo{Ho aggiunto "several"} 
%\begin{AG} Ho creato una sezione, aggiunto  la relazione con la regolarit\`a  e qualche esempio.
%\end{AG}
%\begin{AL}(NUOVA)
%Ottimo!
%\end{AL}
We would like to thank the referee for insightful comments.

\section{Examples}\label{esempi}
%\noindent{\bf Examples.}
Let $X$ be a projective variety and let $H$ be a very ample line bundle. Recall the definition of Castelnuovo-Mumford regularity:
\begin{defn} $H$ is $m$-regular if 
%\begin{AL}(NUOVA)
%Ho corretto
%\end{AL}
%$H^q(X, L^{m+1-q})=0$ for all $q>0$.
$H^q(X, (m+1-q)H)=0$ for all $q>0$.
\end{defn}

\begin{prop}\label{conditions}  Let $X$ be a threefold with klt singularities 
%\begin{AG} klt singularities \end{AG} \begin{AL}(NUOVA2)OK\end{AL}
%\begin{AL}(NUOVA)
%Cohen-Macaulay 3-fold. Ci serve complete?
%\end{AL}
%Cohen Macauley variety
and let $H$ be a very ample line bundle. 
Then $H$ is $0$-regular if and only if  $h^1(\mathcal O_X)=0$ and  $H^0(\omega_X(2H))=0$.
\end{prop}
\begin{proof}  Since klt singularities are  Cohen-Macaulay, by 
Serre's duality we have  $H^3(-2H)=H^0(\omega_X(2H))$ and $H^2(-H)=$ $H^1(\omega_X(H))$; however $H^1(\omega_X(H))=0$ by  Kawamata-Viehweg%-Nadel
's vanishing theorem \cite{Fujino2011}.
%\begin{AL}(NUOVA)
%Non capisco come applicare Kawamata-Viehweg-Nadel 
%senza assumere qualcosa sulle singolart\`a, tipo klt....
%\end{AL}
%\begin{AG}{NUOVA} Ho aggiunto  klt, con cui abbiamo il teorema classico ed andiamo sicuri; Fujino, in 3.2  per Nadel scrive solo normale  e projettiva ed anche EinDeFernexMustata, , ma ora che ho letto la loro dimostrazione non mi convince, quella di Nadel per L2 non l'ho letta, le applicazioni nel caso algebrico, assumono liscia, per "semlplficare".\end{AG}
%\begin{AL}(NUOVA2)
%Si', ok, X normale  e proiettiva, per\`o serve avere una coppia $(X,B)$ con l'ideale moltiplicatore $\I(X,B)$ banale (ovvero singolarit\`a klt - nel nostro caso (X,0)), altrimenti non si ha il vanishing $H^1(\omega_X(H))=0$ ma $H^1(\I(X,B)(\omega_X(H)))=0$
%\end{AL}
\end{proof}
 
\begin{prop} \label{1reg}  Let $X$ be a  normal %\begin{AL}(NUOVA)
%Ci serve complete?
%\end{AL}
irreducible threefold with rational singularities and let $H$ be a very ample line bundle. The hypotheses 
%\begin{AL}(NUOVA)
%Ho aggiunto (i)-(iii)
%\end{AL}
(i)-(iii) of Theorem \ref{esistenza} are satisfied if and only if 
$H$ is 1-regular and $h^1(\mathcal O_X)=0$.
\end{prop}
\begin{proof} The condition $q=1$ for 1-regularity is (ii) of Theorem \ref{esistenza},  $q=2$ is  the first  part of (i) and $q=3$ becomes (iii) with Serre's duality. 
%Serre duality  in fact applies because rational singularities are Cohen-Macauley.
\end{proof}
\begin{prop} Let $X$ be a  normal irreducible threefold and let $H$ be a very ample line bundle. If $H$ is 0-regular then the hypotheses 
%\begin{AL}(NUOVA)
%Ho aggiunto (i)-(iii)
%\end{AL}
(i)-(iii) of Theorem \ref{esistenza} are satisfied.
 \end{prop}
\begin{proof} $H^1(\mathcal O_X)=0$ is the 0-regularity condition for $q=1$, and  we  conclude by Proposition \ref{1reg}, since $H$ is also 1-regular as it is $0$-regular. \end{proof}
Note that  many varieties with mild singularities are Cohen-Macaulay, such as the ones with klt singularities or normal toric varieties \cite{Kollar-Mori}.

\begin{example}\label{WP}
 The weighted projective spaces %\todo{Ho ulteriormente compattato la lista} 
%\todo{Aggiornati gli esempi}
\begin{itemize}  \item[(2.1.1)] The infinite series  $\mathbb P[1,1,1,q]$, $\mathbb P[1, 2, 2q-1,2q-1]$ and $\mathbb P[2,2,2q-1,2q-1]$, $\ q \in \N$
\item[(2.1.2)]  $\mathbb P[1,1,2,3]$, $\mathbb P[3,3,4,4]$, $\mathbb P[3,3,5,5]$, $\mathbb P[1,2,2,3]$
\end{itemize}
satisfy the hypotheses of Theorem \ref{esistenza}. In fact, 
let $\mathbb P_{\Sigma} = \mathbb P[q_0,q_1,q_2,q_3]$ be a weighted projective 3-space with reduced weights $\{q_i\}$ \cite{Dolgy} and  let $\eta_0$  be the effective generator of the class group of $\mathbb P_{\Sigma}$.  
Then   $\eta=\delta\eta_0$ is the very ample generator of the Picard group $\operatorname{Pic}(\mathbb P _{\Sigma})$, and $\sigma\eta_0=-K_{\mathbb P_{\Sigma}}$ is the anti-canonical class, where $\delta = \operatorname{l.c.m.}(q_i)$, and $\sigma=\sum_iq_i$.  
The 3-fold $\mathbb P_{\Sigma}$ is normal, $\QQ$-factorial, irreducible, it has rational singularities, and satisfies conditions %(i), (ii) and (v) in Theorem \ref{esistenza} and also condition (iv) for $d$ big enough.  
(i), (ii) in Theorem \ref{esistenza} and also condition (iv) for $d$ big enough.  
If we take $H=\eta$, condition (iii) is equivalent to $\delta<\sigma$
%\begin{AL}(NUOVA)
%Ho aggiunto una frase per spiegare come vengono fuori gli esempi
%\end{AL}
and this is satisfied precisely in the cases (2.1.1) and (2.1.2).

Note that  $\mathbb P[1,1,1,2]$ and  $\mathbb P[1,1,2,2]$ also satisfy the hypotheses of Corollary \ref{densita'-torica}.
\end{example}

\begin{remark}
%With the same notation as in the above example, we observe that
Cox in   \cite{CoxNoether}   studies the locus $\NL(L)$ where $L = 6q \eta_0$ in $X = \mathbb P[1, 1, 2q,3q]$, with $q \geq 3$. The results of \cite{CoxNoether}  and this paper are somewhat complementary. 
In fact, the starting point of our analysis, as well as \cite{CoxWeighted}, is that  $L$  should be of high enough degree in $X$  to assure that $\rho(S)= \rho(X)$ %=1$  
for a very general  $S \in |L|$  (condition (iv) of Theorem \ref{esistenza}).
When $X=\mathbb P[1, 1, 2q,3q]$  and $L=6q\eta_0$ this  condition can only be satisfied  when $q=2$ since $\omega_X(L) = (q-2) \eta_0$ is not globally generated except when $q=2$. Cox considers instead the case when $q \geq 3$ and proves directly in Proposition 3.2 that the general surface $S$ has $\rho(S)=1$.
 The minimal resolution of a  surface  $S \in |L|$ is a regular elliptic surface $\tilde S$ with   $h^0(\tilde S,K_{\tilde S}) \geq 2$ and a section; the section is the exceptional divisor of the minimal resolution $\tilde S \to S$.
Cox proves then that all but one component of $\NL(L)$ are of maximal codimension and that they are dense in the natural topology. 

The case $q=2$ corresponds to %the case of 
elliptic K3 surfaces with section, for which all the %NL 
components of the Noether-Lefschetz locus have codimension $1$ and are known to be dense.
Our methods do not apply to this case:  our construction depends in particular  upon finding a suitable curve $C \subset X$  and a very ample line bundle $H$ such that condition (v)  in Lemma \ref{lemmetto} is satisfied; but  there is no  such  line bundle $H$ when $q=2$.

 Note also that for  $q=1$,  the system $|L|=|6\eta_0|$ corresponds to  rational elliptic surfaces and   $\rho(S) = 9 > \rho(X)$ for every $S$. 
 In the Example \ref{WP} above we consider instead $X = \mathbb P[1, 1, 2,3]$  and    $L=k\eta_0$ with $k > 6$.
 
 %With the same notation as in the above example, we observe that  \cite{CoxNoether}   studies the locus $\NL(L)$ where $L = 6q \eta_0$ in $X = \mathbb P[1, 1, 2q,3q]$ (that is, the family of regular elliptic surfaces with a section). It is proved that all but one component of $\NL(L)$ are of maximal codimension and that they are dense in the natural topology. The methods of the present paper do not allow to recover this result. In fact, as shown in Lemma \ref{lemmetto}, our construction depends upon finding a suitable curve $C \subset X$. In particular we need the conditions (v) and (vi) to be satisfied. Now $\omega_X(L) = (q-2) \eta_0$ is not globally generated except when $q=2$ but in that case there is no very ample line bundle $H$ such that (v) is satisfied.
\end{remark}

\begin{example}  The quasi-Fano variety  $\PP_\Sigma$, which is the resolution of  the cone over a quadric surface in $\mathbb P^3$, also satisfies the hypotheses of Theorem  \ref{esistenza}. In addition  for any  $ \ d \geq 0$  the bounds 
${ d \leq \codim \NL(- K_{\PP_{\Sigma}} + d H) \leq h^0(dH)}$ are also satisfied \cite{bg2}.
\end{example}

\begin{example}Other examples are provided by Fano varieties. Indeed, using \cite[Thm.~7.80 (c)]{ShiffSomm},
%\todo{Ho sostituito la vecchia referenza [AJ] con \cite[Thm.~7.80 (c)]{ShiffSomm}}
%\cite[Prop.\ 1.1 and Rmk.\ 1.2]{aj}, 
it is easily seen that the hypotheses of Theorem \ref{esistenza} and Corollary \ref{densita'} are satisfied by a normal, $\QQ$-factorial, irreducible Fano threefold with rational singularities $X$ having a very ample line bundle $H$ such that $H^0(K_X+H) = 0$ and $\omega_X(dH)$ is globally generated. In particular this happens when $-K_X = r H$ with $r \ge 2$ and $d \ge r$.

Smooth Fano threefolds of index $2$ were classified by \cite{Wisniewski}.  Fano threefolds with high index and singularities are studied in  \cite{Fujita} and \cite{ProkhorovReid}.
\end{example}
\begin{example}$\mathbb P^1 \times \mathbb P^1 \times \mathbb P^1$ is such a Fano manifold of index 2. In addition, with the methods of Section 5 in \cite{bg2}  (Proposition 5.2 and  Lemma 5.3) we   find that  the codimension of  the smooth surfaces in 
	%$(- K_{\PP_{\Sigma}} + d H)$ 
	$|- K_{\PP_{\Sigma}} + d H|$  for $ d \geq 0$ which contain any of the rulings of  $ \mathbb P^1 \times \mathbb P^1 \times \mathbb P^1$ is $d+1$.
	%\begin{AL}(NUOVA)
	%Ho corretto $(- K_{\PP_{\Sigma}} + d H)$ con $|- K_{\PP_{\Sigma}} + d H|$. A me viene $d$ e non $d-1$ e poi aggiungerei "for $d \ge 0$".
	%\end{AL} \begin{AG}{nuova} dovrebbe essere d+1,  Aggiunto la citazione precisa\end{AG}
	%\begin{AL}(NUOVA2) 
	%Giusto!
	%\end{AL}
\end{example}
\begin{example}$\mathbb P^1\times \mathbb P^2$ 
%satisfy 
satisfies the hypotheses of Corollary \ref{densita'-torica} with $H= p_{\mathbb P^1}^*(O_{\mathbb P^1}(1)) \oplus p_{\mathbb P^2}^*(O_{\mathbb P^2}(1))$.  Moreover,  the bounds $ d \leq \codim \NL(- K_{\PP_{\Sigma}} + d H) \leq h^0(dH) $ are satisfied, for $d \ge 0$  \cite{bg2}.
%\begin{AL}(NUOVA)
%A me il lower bound viene $d$ e non $d-1$ e poi aggiungerei "for $d \ge 0$".
%\end{AL}
\end{example}

\begin{example} A rational threefold with $\mathbb Q$-factorial klt %\begin{AG} klt \end{AG} 
singularities and $-K_X$ nef %semiample  
%\begin{AL}(NUOVA2) basta nef
%\end{AL}
satisfies the hypotheses %\begin{AL}(NUOVA2) Ho aggiunto (i)-(iii)\end{AL}
(i)-(iii) of Theorem \ref{esistenza} if  $H^0(\omega_X(H))=0$, as Kawamata-Viehweg's vanishing theorem applies. 
%\begin{AL}(NUOVA2) Ho percentualizzato la frase che c'era dopo, in effetti \`e ovvio che si applica KV. Mi ero confuso tra $K_X$ e $-K_X$.
%\end{AL}
%In fact $-K_X+H$ is nef and big, actually very ample, since $-K_X$ is semiample and $H$ very ample.
%\begin{AG}(NUOVA) Aggiunto l'ultime riga, in riposta al commento di Angelo: \end{AG}
%\begin{AL}(NUOVA)
%Per applicare KV non ci serve $K_X$ big e nef?
%\end{AL}
\end{example}
%\begin{AL}(NUOVA2)
%Nell'esempio che segue ho fatto varie piccole correzioni, typo, ecc.
%\end{AL}
\begin{example}  The projective 3-space blown-up along a line $\widehat{\mathbb P}^3$ is one such example. The nef cone is generated by $\eta_1$, the pullback of a plane in $\mathbb P^3$ and $\eta _2= \eta_1 - E$, where $E$ is the exceptional divisor. Any $H=s_1\eta_1+ s_2 \eta_2 $ with $ s_1=1, 2, \  s_2  \geq 1$ is  very ample, while $h^0(K _{\widehat{\mathbb P}^3}+H)=0$. Also $-K_ {\widehat{\mathbb P}^3} + H$ is very ample, $H$ is 0-regular and the hypotheses of Theorem \ref{esistenza} are satisfied for $s_1=1, d \ge 3$ and $s_1=2, d \ge 2$. Moreover, we have the bounds  \cite{bg2} $$d \leq \codim \NL(-K_{\PP_{\Sigma}} + d H) \leq h^0(dH).$$
Note, however, that  $-K_{\PP_{\Sigma}} -2 H$ is not nef and thus the hypotheses of  Corollary \ref{densita'-torica} are not satisfied; in fact the cone of effective divisors includes the nef cone.
%\begin{AG}(NUOVA): Aggiunto l'ultima riga per chiarire: \end{AG}\begin{AL}(NUOVA)
%Secondo me $\widehat{\mathbb P}^3$ non soddisfa l'ipotesi $-K_{\PP_{\Sigma}} -2 H$ nef. Nella notazione del vostro articolo mi viene $-K_{\PP_{\Sigma}} -2 H = \eta_1 + (1-2s) \eta_2$, che non \`e nef se $s \ge 1$.
%\end{AL}
\end{example}
%VECCHIO ESEMPIO
%\begin{example}  The projective 3-space blown-up along a line $\widehat{\mathbb P}^3$ is one such example. The nef cone is generated $\eta_1$, the pullback of a line in $\mathbb P^3$ and $\eta _2= \eta_1 - E$, where $E$ is the exceptional divisor. Any $H=s_1\eta+ s_2 \eta_2 $ with $ s_1=1, 2 \  s_2  \geq 1$ is  very ample, while $h^0(K _{\widehat{\mathbb P}^3}+H)=0$. Aslo $-K_ {\widehat{\mathbb P}^3} + H=$ is very ample and the hypothesis of Theorem \ref{esistenza} are satisfied. . Moreover, if $s_1=1, 2$, $H$ is 0-regular and we have the bounds:   $ d \leq \codim \NL(- K_{\PP_{\Sigma}} + d H) \leq h^0(dH) $ \cite{bg2}.
%Note however that  $-K_{\PP_{\Sigma}} -2 H$ is not nef and thus the hypothesis of  Corollary \ref{densita'-torica} are not satisfied; in fact the cone of effective divisors includes the nef cone.
%\begin{AG}(NUOVA): Aggiunto l'ultima riga per chiarire: \end{AG}\begin{AL}(NUOVA)
%Secondo me $\widehat{\mathbb P}^3$ non soddisfa l'ipotesi $-K_{\PP_{\Sigma}} -2 H$ nef. Nella notazione del vostro articolo mi viene $-K_{\PP_{\Sigma}} -2 H = \eta_1 + (1-2s) \eta_2$, che non \`e nef se $s \ge 1$.
%\end{AL}
%\end{example}

\section{Existence and maximal codimension of components}
\label{comp}
%\todo{Ho aggiunto la frase "Unless otherwise specified,". Infatti, in alcuni punti, X \`e diversa...}
Unless otherwise specified, throughout this paper $X$ will be a normal complex 
%\todo{Ho aggiunto "complex"}  
$\QQ$-factorial irreducible threefold with rational singularities. We shall denote by $\omega_X$ its dualizing sheaf.

When $X$ is smooth, there are well-known conditions that assure the existence of components of the Noether-Lefschetz locus, namely that $h^{2,0}_{ev}(S,\CC) > 0$ for $S \in |L|$ general \cite{mo}, \cite[Thm.\ 15.33]{v3}. If $X$ is a toric threefold, 
the same is assured by  a suitable  combinatorial condition \cite{BG1}. 

\begin{rem} \hskip 5cm
%\todo{Ho precisato alcune referenze}
\begin{enumerate} 
\item Since $X$ has rational singularities, it is Cohen-Macaulay, and $p_{\ast} \omega_{\overline{X}} \simeq \omega_X$, where $p : \overline{X} \to X$ is any desingularization \cite[Thm.~5.10]{Kollar-Mori}.
\item For every projective normal variety $X$ with  rational singularities,  the group $H^2(X,\ZZ)$ has a pure  Hodge structure
induced by that of a desingularization \cite[Lemma 2.1]{Bakker-Lehn}, \cite{Sri}.
\item The general hyperplane section of a variety with rational singularities has rational singularities \cite[Rmk.~3.4.11(3)]{Flenner-O'Carroll-Vogel}, and  a general hyperplane section of a normal variety is normal \cite[Thm.~7']{Sei}.\end{enumerate}
\label{remarks}
\end{rem}
%\begin{AG}Aggiunto (i) perch\'e Koll\'ar quando ho fatto seminario mi ha chiesto  ed l'ha trovato interessante. %\end{AG}
%\begin{AL}(NUOVA)
%OK
%\end{AL}

\begin{prop}
\label{codmax}
Let $X$ be as above,
%a normal $\QQ$-factorial irreducible threefold with rational singularities, 
and let $L$ be a very ample line bundle on $X$. Assume that $\omega_X(L)$ is globally generated. Then:

\begin{enumerate}
 \item $\Cl(X) \simeq \Cl(S)$,  for a very general $S \in |L|$.
\item $\rho(S) = \rho(X)$ for a very general $S \in |L|$ (thus one can define  the Noether-Lefschetz locus $\NL(L)$).
\item For every component $V$ of $\NL(L)$, and for every $S \in V$, we have  
\[ \codim_{U(L)} V \le h^{2,0}(S) = h^0(\omega_X(L)) + h^2(\O_X) - h^3(\O_X). \] 
\end{enumerate}
\end{prop}
\begin{proof} (i) Let $f\colon X \to \mathbb P^N$ be the embedding given by the line bundle $L$. It was shown in \cite[Thm.\ 1]{rs} 
that $\Cl(X) \simeq \Cl(S)$ for a very general surface $S$  in $\vert L\vert$ whenever the line bundle %$f_\ast(\omega_X)(1)$ 
$(f_\ast \omega_X)(1)$ is globally generated. To show that 
%\todo{Ho sostituito "the" con "that"}  
this condition holds, we write the exact sequence 
$$ H^0(X, \omega_X(L))\otimes \mathcal O_X \to \omega_X(L) \to 0 .$$
We apply the functor $f_\ast$ obtaining a surjective morphism $ H^0(X, \omega_X(L)) \otimes f_\ast  \mathcal O_X \to (f_\ast \omega_X)(1)$, %f_\ast(\omega_X)(1)$, 
and, by composing with the evaluation morphism $\mathcal O_{\mathbb P^N} \to f_\ast  \mathcal O_X$, we obtain a surjective morphism $H^0(X, \omega_X(L)) \otimes \mathcal O_{\mathbb P^N} \to (f_\ast \omega_X)(1)$. %f_\ast(\omega_X)(1).$
%\begin{AL}(NUOVA)
%Ho aggiunto la frase sotto
%\end{AL}
%So we have  .
Hence %$f_\ast(\omega_X)(1)$ 
$(f_\ast \omega_X)(1)$ is globally generated. 
%\begin{AL}(NUOVA)
%Ho aggiunto (ii)
%\end{AL}

(ii) Since $S$ is normal, we have two injections $\Pic(X) \hookrightarrow \Pic(S)$ (as in the Introduction), and $\Pic(S) \hookrightarrow \Cl(S)$, whence, using the $\QQ$-factoriality of $X$, we get 
\begin{eqnarray*} \rho(X) &\le& \rho(S)  =  \rk (\Pic(S) \otimes \QQ) \le \rk (\Cl(S) \otimes \QQ) \\ &=& \rk (\Cl(X) \otimes \QQ) = \rk (\Pic(X) \otimes \QQ) = \rho(X). \end{eqnarray*}
%\begin{AL}(NUOVA)
%Ho sostituito (ii) con (iii)
%\end{AL}

%(ii) 
(iii) Now let $V$ be a component of $\NL(L)$ and let $S \in V$, so that $\rho(S) > \rho(X)$.
In the smooth case, as is well known \cite[pages 71-72]{cggh}, this gives $h^{2,0}(S)$ conditions.
By Remark \ref{remarks} (ii), one can reason as in \cite[Prop.~4.6]{bg2} and obtain, using \cite[Thm.~7.80 (c)]{ShiffSomm},
%\todo{Ho aggiunto "using \cite[Thm.~7.80 (c)]{ShiffSomm}"}
\[ \codim_{U(L)} V \le h^{2,0}(S) = h^0(\omega_X(L)) + h^2(\O_X) - h^3(\O_X). \]
%By \cite[Prop.\ 1.1 and Rmk.\ 1.2]{aj} we have that $H^2(- L) = 0$, whence the exact sequence
%\begin{equation}
%\label{esse}
%0 \to - L \to \O_X \to \O_S \to 0\,.
%\end{equation}
%Since 
%$$ h^{2,0}(S) =  h^{2,0} (\overline{S}) =  h^{0,2} (\overline{S}) =  h^{0,2}(S)$$
%and $$H^2(\overline{S},\mathcal O_{\overline{S}}) \simeq H^2(S,\mathcal O_S) $$
%as $S$ has rational singularities, 
%we have  $h^{2,0}(S) = h^2(\O_S)$. Then from \eqref{esse}, using Serre duality, we have 
%$$ h^{2,0}(S)
% = h^0(\omega_X(L)) + h^2(\O_X) - h^3(\O_X).$$
\end{proof}

\section{Components of maximal codimension from curves}
\label{curv}

In the case of $\PP^3$, components of maximal codimension have been constructed in two ways: by a degeneration argument in \cite{CilibertoHarriaMirandaNLLocus}, and by choosing suitable components of the Hilbert scheme in \cite{cl}. We consider here the second approach. 

We first show that we can construct components of maximal codimension as soon as we have some curve in $X$ with 
good properties.
%\begin{AL}
%Rispetto alla versione precedente, ho tolto la vecchia (ii) e riscalato le altre.
%\end{AL}

\begin{lemma}
\label{lemmetto}
Let $X$ be as above,  and let $L$ be a very ample line bundle on $X$. Let $W$ be a component of the Hilbert scheme of curves on $X$ such that there is a smooth irreducible curve $C$ representing a point of $W$, and with $C \cap \Sing(X) = \emptyset$. Moreover, suppose that:
\begin{enumerate}
\item $H^i(\O_X) = 0$ for $i>0$;%\todo{Ho messo $H^i$ invece di $h^i$}
\item  $H^1(N_{C/X}) = 0$;
\item $H^1(\I_{C/X}(L)) = 0$;
\item $H^0(\I_{C/X} \otimes \omega_X(L)) = H^1(\I_{C/X} \otimes \omega_X(L)) = 0$;
\item there is a very ample line bundle $H$ on $X$ such that $\I_{C/X}(L - H)$ is globally generated;
\item $\omega_X(L)$ is globally generated.
\end{enumerate}
Then $W$ defines a component $W(L)$ of maximum codimension of $\NL(L)$, that is,
\begin{equation}
\label{cod}
\codim_{U(L)} W(L) = h^0(\omega_X(L)).
\end{equation}
\end{lemma}
%VERSIONE VECCHIA
%\begin{lemma}
%\label{lemmetto}
%Let $X$ be as above,  and let $L$ be a very ample line bundle on $X$. Let $W$ be a component of the Hilbert scheme of curves on $X$ such that there is a smooth irreducible curve $C$ representing a point of $W$, and with $C \cap \Sing(X) = \emptyset$. Moreover, suppose that:
%\begin{enumerate}
%\item $H^i(\O_X) = 0$ for $i>0$;%\todo{Ho messo $H^i$ invece di $h^i$}
%\item $H^1(L_{|C})=0$;
%\item  $H^1(N_{C/X}) = 0$;
%\item $H^1(\I_{C/X}(L)) = 0$;
%\item $H^0(\I_{C/X} \otimes \omega_X(L)) = H^1(\I_{C/X} \otimes \omega_X(L)) = 0$;
%\item there is a very ample line bundle $H$ on $X$ such that $\I_{C/X}(L - H)$ is globally generated;
%\item $\omega_X(L)$ is globally generated.
%\end{enumerate}
%Then $W$ defines a component $W(L)$ of maximum codimension of $\NL(L)$, that is,
%\begin{equation}
%\label{cod}
%\codim_{U(L)} W(L) = h^0(\omega_X(L)).
%\end{equation}
%\end{lemma}
\begin{proof}
By (vi) we can apply Proposition \ref{codmax}, that is, the components of the Noether-Lefschetz locus $\NL(L)$ exist. Note that $\I_{C/X}(L)$ is globally generated by (v). Let $S \in |\I_{C/X}(L)|$ be very general. We claim that:
\begin{itemize} \item[a)]   the conditions 
\begin{equation}
\label{rango}
S \in U(L) \ \mbox{and} \ \rho(S) = \rho(X)+1
\end{equation} 
hold; 
\item[b)] the same conditions of the Lemma and (a) %\todo{Ho aggiunto "of the Lemma and (a)"} 
hold for a curve $C_\eta$ representing a generic point in $W$, and a very general surface $S_\eta$ in the linear system $ |\I_{C_\eta/X}(L)|$.
\end{itemize} 
To  prove this let $\pi : \widetilde{X} \to X$ be the blow-up of $X$ along $C$ with exceptional divisor $E$, and let $\widetilde{S}$ be the strict transform of $S$, so that $\widetilde{S} \simeq  S$. Note that $\widetilde{L} := \pi^{\ast}L - E = \pi^{\ast}(L - H) - E + \pi^{\ast}H$ is very ample by (v) (and, for example, \cite[4.1]{ps} 
%\begin{AL}
%Qui ho aggiunto la referenza a BS
%\end{AL}
or \cite[Proof of Thm.2.1]{bs}).
%\todo{Verificare rif. PS}
Since $\widetilde{S}$ is general in $\widetilde{L}$ and $\widetilde{X}$ is also normal with rational singularities, it follows that $\widetilde{S}$ is irreducible, normal with rational singularities, whence so is $S$, and therefore $S \in U(L)$. 
Now $\omega_{\widetilde{X}}(\widetilde{L})$ is globally generated by (vi), and moreover, 
$\Cl(\widetilde{X}) \simeq \ZZ E \oplus \Cl(X)$; %\todo{non vogliamo mettere la referenza?} 
thus, as in Proposition \ref{codmax}, we get 
% whence \cite[Thm.\ 1]{rs} implies that $\Cl(\widetilde{X}) \simeq \Cl(\widetilde{S})$. Since $\Cl(\widetilde{X}) \simeq \ZZ E \oplus \Cl(X)$ (see for example \cite[Prop.\ 2.6]{bn}), using $\QQ$-factoriality, we get 
\[ \rho(\widetilde{X}) = \rk (\Cl(\widetilde{X}) \otimes \QQ) = \rk (\Cl(X) \otimes \QQ) + 1 = \rho(X) + 1. \]
Moreover, as $\widetilde{S}$ is normal, we have $\Pic(\widetilde{X}) \hookrightarrow \Pic(\widetilde{S})$ (as in the Introduction), and $\Pic(\widetilde{S}) \hookrightarrow \Cl(\widetilde{S})$, whence
\[ \rho(X) + 1 =  \rho(\widetilde{X}) \le \rho(\widetilde{S}) \le \rk (\Cl(\widetilde{S}) \otimes \QQ) = \rk (\Cl(\widetilde{X}) \otimes \QQ) = \rho(X) + 1 \]
and \eqref{rango}(a) is proved. %\todo{Ho aggiunto "(a)"}

Let $g$ be the genus of $C$. From the exact sequence
\[ 0 \to \I_{C/X}(L) \to L \to L_{|C} \to 0 \]
using (iii) 
%and (ii) 
%\begin{AL}
%Avendo tolto la condizione $H^1(L_{|C})=0$ ora si ottiene una disuguaglianza
%\end{AL}
we get 
\begin{equation}
\label{ide}
h^0(L) - h^0(\I_{C/X}(L)) = h^0(L_{|C}) \ge L \cdot C - g + 1.
\end{equation}
%VERSIONE VECCHIA
%\begin{equation}
%\label{ide}
%h^0(L) - h^0(\I_{C/X}(L)) = h^0(L_{|C}) = L \cdot C - g + 1.
%\end{equation}
%Now let $C_{\eta}$ be a curve in $X$ representing a general point of $W$. Then $C_{\eta}$ is smooth irreducible, 
Now consider $C_{\eta}$. This curve is smooth and irreducible, $C_{\eta} \cap \Sing(X) = \emptyset$, %\todo{Cambio di testo}
and by semicontinuity the conditions (ii)-(iv) hold for $C_{\eta}$. The exact sequence
\[ 0 \to \I_{C_{\eta}/X}(L) \to L \to L_{|C_{\eta}} \to 0 \]
gives, by semicontinuity
%using \eqref{ide},
%\begin{AL}
%Ho dovuto modificare anche la dim. dell'uguaglianza sotto e aggiungere (riga sopra) "by semicontinuity" e (riga sotto) %whence "we get equality"
%\end{AL}
\begin{multline*} h^0(\I_{C/X}(L)) \ge h^0(\I_{C_{\eta}/X}(L)) = h^0(L) - h^0(L_{|C_{\eta}}) \ge h^0(L) - h^0(L_{|C}) = h^0(\I_{C/X}(L)). \end{multline*}
whence we get equality.
%VERSIONE VECCHIA
%\begin{multline*}h^0(\I_{C_{\eta}/X}(L)) = h^0(L) - h^0(L_{|C_{\eta}}) = h^0(L) - L \cdot C_{\eta} + g - 1 \\ = h^0(L) - L \cdot C + g - 1 = h^0(\I_{C/X}(L)). \end{multline*}

Now let $S_{\eta} \in |\I_{C_{\eta}/X}(L)|$ be very general; then \eqref{rango}(a) holds %\todo{Ho aggiunto "(a)"}
for $S_{\eta}$. For ease of notation, in the sequel of the proof we will replace $C_{\eta}$ with $C$ and $S_{\eta}$ with $S$. From (ii) we get
\begin{equation}
\label{hilb}
\dim W = h^0(N_{C/X}) = \chi(N_{C/X}) = \deg N_{C/X} + 2 - 2g = \deg {T_X}_{|C} = - \deg {\omega_X}_{|C}.
\end{equation}
Consider the incidence correspondence
\begin{equation*}
\J = \{ (S', C') : C' \subset S' \} \subset U(L) \times W
\end{equation*}
together with its projections 
\[ \xymatrix{& \J \ \ \ar[dl]_{\pi_1} \ar[dr]^{\pi_2} & \\ \ \ \ U(L) & & W} \]
and let $W(L) = \Im \pi_1$. Now \eqref{rango} implies that $\pi_2$ is dominant, hence, using \eqref{hilb} we find
\begin{eqnarray*}
\dim W(L) & = & \dim \J - (h^0(\O_S(C)) - 1) = \dim W + (h^0(\I_{C/X}(L)) - 1) - (h^0(\O_S(C)) - 1)   \\
& = & - \deg {\omega_X}_{|C} + h^0(\I_{C/X}(L)) - h^0(\O_S(C))
\end{eqnarray*}
whence
\begin{equation}
\label{inc}
\codim_{U(L)} W(L) = h^0(L) - 1- h^0(\I_{C/X}(L)) + \deg {\omega_X}_{|C} + h^0(\O_S(C)).
\end{equation}
%\begin{AL}
%Qui sotto si deduceva, nella versione precedente, che $H^1(\O_S)=0$, rimandando alla dimostrazione della Proposizione \ref{codmax}. Ma questa parte, nella dimostrazione della Proposizione \ref{codmax} ad un certo punto l'abbiamo tolta. Dunque ho rimesso qui sotto la dimostrazione che $H^1(\O_S)=0$.
%\end{AL}
Since $H^2(\O_X(- L)) = 0$ by \cite[Thm.~7.80 (c)]{ShiffSomm}, the exact sequence
$$0 \to \O_X(- L) \to \O_X \to \O_S \to 0$$
and (i) give that $H^1(\O_S)=0$ and then the exact sequence
%As in the proof of Proposition \ref{codmax}, using (i), we deduce that $H^1(\O_S)=0$ and the exact sequence
\[ 0 \to \O_S \to \O_S(C) \to \O_C(C) \to 0 \]
gives
\begin{equation}
\label{os}
h^0(\O_S(C)) - 1 = h^0(\O_C(C)) = h^1({\omega_S}_{|C}) = h^1(\omega_X(L)_{|C})
\end{equation}
(here we use the adjunction formula for $S$ in $X$, see e.g.~\cite[Eq.~4.2.9]{ko}).

Moreover, note that, by the hypothesis $C \cap \Sing(X) = \emptyset$, the following sequence
\[ 0 \to \I_{C/X} \otimes \omega_X(L) \to \omega_X(L) \to \omega_X(L)_{|C} \to 0 \]
is exact, so that, using (iv), we get 
\begin{equation}
\label{pgbis}
h^0(\omega_X(L)) = h^0(\omega_X(L)_{|C}).
\end{equation}
Putting together \eqref{inc}, \eqref{ide}, \eqref{os} and \eqref{pgbis} we have
\begin{eqnarray*}
\codim_{U(L)} W(L) & \ge & L \cdot C - g + 1 + \deg {\omega_X}_{|C} + h^1(\omega_X(L)_{|C}) = \\ 
& = & \deg \omega_X(L)_{|C} - g + 1 + h^1(\omega_X(L)_{|C}) = h^0(\omega_X(L)_{|C}) = h^0(\omega_X(L)).
\end{eqnarray*}
%\begin{AL}
%Qui ho aggiunto un'altra frase
%\end{AL}
It remains to prove that $W(L)$ is a component of $\NL(L)$. This, together with Proposition \ref{codmax}, will give that $\codim_{U(L)} W(L) = h^0(\omega_X(L))$.

Let $V$ be a component of $\NL(L)$ containing $W(L)$ and let $S'$ be a surface representing its general point, so that \eqref{rango} gives  $\rho(S') = \rho(X)+1$. 
%\todo{$W(L)$ \`e chiuso perch\'e $\pi_1$ \`e proprio?}
Then we can assume that there is a line bundle $L'$ on $S'$ that specializes to $\O_S(C)$ when $S'$ specializes, in $V$, to $S$. It will therefore suffice to prove that $h^0(L')=h^0(\O_S(C))$ (so that $L'$ is effective and therefore corresponds to a deformation of $C$). By semicontinuity we have $h^0(L') \le h^0(\O_S(C))$ and $ h^2(L') \le h^2(\O_S(C))$, and then
\begin{equation}
\label{h1}
h^1(L') \le h^1(\O_S(C)) = h^1(\omega_S(-C)) = h^1(\I_{C/S} \otimes \omega_X(L))
\end{equation}
where the last equality follows by the adjunction formula. Now we have an exact sequence
\[ 0 \to \F \to \I_{S/X} \otimes \omega_X(L) \to \I_{C/X} \otimes \omega_X(L) \to \I_{C/S} \otimes \omega_X(L) \to 0 \]
where $\F$ is a sheaf with support of dimension at most $1$. Since $\I_{S/X} \otimes \omega_X(L) \simeq \omega_X$, we get $H^2(\I_{S/X} \otimes \omega_X(L)) = H^2(\omega_X) = H^1(\O_X) = 0$ by (i). Then (iv) %\todo{Ho sostituito "(vi)" con "(v)"} 
gives $h^1(\I_{C/S} \otimes \omega_X(L)) = 0$, so that $h^1(L') = 0$ by \eqref{h1}. Therefore
\begin{multline} h^0(L') = \chi (L') + h^1(L') - h^2(L') = \chi(\O_S(C)) - h^2(L') \\ \ge \chi(\O_S(C)) - h^2(\O_S(C)) = h^0(\O_S(C)) 
\end{multline}
and we are done.
\end{proof}

Now we shall see %\todo{Ho aggiunto "see"} 
how the conditions in Lemma \ref{lemmetto} can be met. To get condition (ii) of Lemma \ref{lemmetto} we will adapt a result of Kleppe \cite{kl}.

\begin{lemma}
\label{kleppe}
Let $X$ be a Cohen-Macaulay projective threefold such that $H^i(\O_X) = 0$ for $0 < i < 3$. Let $\Gamma$ be a Cohen-Macaulay equidimensional subscheme of $X$ of dimension $1$ such that $X$ is smooth along $\Gamma$. Then
\[ H^1(N_{\Gamma/X}) \simeq \Ext^2_{\O_X}(\I_{\Gamma/X}, \I_{\Gamma/X}). \]
\end{lemma}
\begin{proof}
We apply \cite[Remark 2.2.6]{kl}. Setting, in Kleppe's notation, $\PP = X$ and $X=\Gamma$, we need to satisfy the conditions in \cite[Thm.\ 2.2.1]{kl}, with the exception of the requirement that $\Gamma$ is generically complete intersection. Hence it suffices to verify that there is an embedding $X \subset \PP^N$ such that the cone is Cohen-Macaulay. Since $H^i(\O_X) = 0$ for $0 < i < 3$, this can be obtained via a sufficiently ample embedding, in the following, probably well-known, way. Let $H$ be very ample on $X$. Then there exists $m_1 \in \NN$ such that $H^i(mH) = 0$ for $i > 0$ and $m \ge m_1$. By Serre duality there exists $m_2 \in \NN$ such that $H^i(-mH) = 0$ for $i < 3$ and $m \ge m_2$. Moreover, let $m_3 \in \NN$ be such that $S^k H^0(mH) \twoheadrightarrow H^0(kmH)$ for every $k \in \NN$ and for every $m \ge m_3$. Then, setting $m_0 = \max\{m_1, m_2, m_3\}$, and embedding $X \subset \PP^N = \PP H^0(m_0H)$ we have that $H^i(\O_X(j)) = 0$ for every $j \in \ZZ$ and for all $i$ such that $0 < i < 3$. 
%Moreover, the homogeneous coordinate ring of $X$ is $\bigoplus_{m \ge 0} H^0(\O_X(m))$, whence Cohen-Macaulay.
Now we can apply Corollary 3.11 in \cite{ko}.
\end{proof}
%\todo{Cosa possiamo dire per: two locally free sheaves on $X$ such that $\rk(\mathcal F) = \rk(\mathcal E)+1, \det \mathcal E \simeq \det \mathcal F$ and $\mathcal E^* \otimes \mathcal F$  ample and generated by global sections?} 
Next, to construct curves having the properties of Lemma \ref{lemmetto}, we use degeneracy loci of morphisms of vector bundles.

\begin{prop} 
\label{minori}
Let $X$ be a normal projective irreducible threefold, let $H$ be a very ample line bundle on $X$ and let $\mathcal E = \O_X(-dH)^{\oplus (d-1)}, \mathcal F = \O_X((1-d)H)^{\oplus d}$ for $d \ge 2$. %\todo{Ho aggiunto "for $d \ge 2$"}
Let $\phi : \mathcal E \to \mathcal F$ be a general morphism and let $C = D_{\rk(\mathcal E)-1}(\phi)$ be its degeneracy locus. Then $C$ is a smooth irreducible curve such that $C \cap \Sing(X) = \emptyset$. 
\end{prop}
\begin{proof} By \cite[Thm.~2.8]{ott} or \cite[Thm.~1]{Banica} and \cite[Thm.\ II]{fl} we see that $C$ is a smooth irreducible curve. We need to prove that $C$ does not pass though  $\Gamma$, the singular locus of $X$. Note that $ \dim (\Gamma) \leq 1$. Recall that a general morphism $\phi : \mathcal E \to \mathcal F$ is represented by a $(d, d-1)$ matrix $M_d$ with general entries %$\{\Phi_{i,j}\}  \in H^0(X, H)$. 
$\Phi_{i,j} \in H^0(X, H)$.
%\todo{Secondo me le parentesi graffe non servono qui}  

For %$i = 1, \cdots, d$
$i = 1, \ldots, d$ %\todo{ldots}
let $F^d_i$ be hypersurface on $X$ defined by the minor $D^d_i$ of $M_d$ obtained by removing the $i$-th row. We will prove, by induction on $d$, that for a general $M_d$
\begin{equation}
\label{int2}
F^d_{d-1} \cap F^d_d \cap \Gamma = \emptyset.
\end{equation}
Equation \eqref{int2} proves that $C \cap \Sing(X) = \emptyset$ since $C \subseteq F^d_{d-1} \cap F^d_d$.

If $d=2$, $D^2_1 = \Phi_{2,1}, D^2_2 = \Phi_{1,1}$ whence \eqref{int2} holds since $H$ is very ample and $\Phi_{1,1}, \Phi_{2,1}$ are general.

Next suppose $d \ge 3$ and that \eqref{int2} holds for $M_{d-1}$. Then it clearly also holds for the $(d-2,d-1)$ transpose matrix $M_{d-1}^T$, that is 
%\todo{ho corretto gli indici nella (9), che sono d-2 e d-1}
\begin{equation}
\label{int3}
%F^{d-1}_{i} \cap F^{d-1}_{j} \cap \Gamma = \emptyset
F^{d-1}_{d-2} \cap F^{d-1}_{d-1} \cap \Gamma = \emptyset
\end{equation}
where  $F^{d-1}_i$ %becomes 
is the hypersurface defined by the minor $D^{d-1}_i$ of $M_{d-1}^T$ obtained by removing the $i$-th column. 

Let $M_d$ be the $(d, d-1)$ matrix obtained by adding to $M^T_{d-1}$ two  bottom rows with general entries %$\{\Phi_{d,j}\} $ and  $\{\Phi_{d+1,j}\} $ in $H^0(X, H)$. 
$\Phi_{d-1,j}$ and $\Phi_{d,j}$ in $H^0(X, H)$.
%\todo{Secondo me le parentesi graffe non servono qui. Gli indici invece sono d-1 e d}

The ${({d-1}, d-1)}$ minors $D^d_{d-1}$  and  $D^d_{d}$ of $M_d$ can be computed as:
\begin{equation}
\label{int4}
D^d_{d-1} = \sum\limits_{i=1}^{d-1} (-1)^{i+d-1} D^{d-1}_i \Phi_{d,i} \  \ \mbox{and } \   \  \  D^d_d = \sum\limits_{i=1}^{d-1} (-1)^{i+d-1} D^{d-1}_i \Phi_{d-1,i}.
\end{equation}

%Now, for every $(a_1, \ldots, a_{d-1}) \in \CC^{d-1}$ with $(a_1, \ldots, a_{d-1}) \neq (0,\ldots,0)$, set
%%\todo{Ho rimpiazzato "there exists" con $\exists$ in modo che non esce dal margine}
%$$V{(a_1, \ldots, a_{d-1}) }= \{s \in H^0(X, H) : \exists \ s_1,\ldots, s_{d -1}\in H^0(X,H) \ \mbox{with} \ s = a_1s_1+\ldots + a_{d-1} s_{d-1}\}.$$

%It is clear that $V{(a_1, \ldots, a_{d-1})} = H^0(X,H)$,  in particular  $V{(a_1, \ldots, a_{d-1})} $ is a base-point free linear system. 

Note that for every $x \in \Gamma$ it follows by \eqref{int3} that 
%\todo{Ho aggiunto "it follows by \eqref{int3} that" e sono andato a capo, per maggior chiarezza}
$$((-1)^{1+d-1}D^{d-1}_1(x), \ldots, (-1)^{d-1+d-1}D^{d-1}_{d-1}(x)) \neq (0,\ldots,0),$$  
whence %$V((-1)^{1+d-1}D^{d-1}_1(x), \ldots, (-1)^{d-1+d-1}D^{d-1}_{d-1}(x))=H^0(X,H)$ 
setting $a_i=(-1)^{i+d-1}D^{d-1}_i(x)$, the linear system
$$V(x) := \{s \in H^0(X, H) : \exists \ s_1,\ldots, s_{d -1}\in H^0(X,H) \ \mbox{with} \ s = a_1s_1+\ldots + a_{d-1}s_{d-1}\}$$
is the whole $H^0(X,H)$, whence base-point free. Therefore, choosing general %$\{\Phi_{d,i}\}$'s 
$\Phi_{d,i}$'s 
%\todo{Secondo me le parentesi graffe non servono qui} 
and using \eqref{int4}, we see that the hypersurface $F^d_{d-1}$ does not contain $\Gamma$ and will therefore intersect $\Gamma$ at finitely many points $\{x_1, \ldots, x_s\}$. 

Again by \eqref{int3} the linear systems %$V((-1)^{1+d-1}D^{d-1}_1(x_k), \ldots, (-1)^{d-1+d-1}D^{d-1}_{d-1}(x_k))$ 
$V(x_k)$ are base-point free for every $1 \le k \le s$, whence choosing general $\Phi_{d-1,i}$'s and using \eqref{int4}, we see that $D^d_d(x_k) \neq 0$ for all $k$, that is $x_k\not\in F^d_d$. This proves \eqref{int2}.
Note that a  linear algebra argument shows also that 
\begin{equation}
\label{int1}
F^d_{i} \cap F^d_j \cap \Gamma = \emptyset, \ \ \forall i, j.
\end{equation}
\end{proof}
%\bigskip
%\begin{AL}
%Rispetto alla versione precedente, ho tolto la vecchia (b) e riscalato le altre. Inoltre le condizioni $H^2(\mathcal F(L)) = 0$ e $H^3(\mathcal E(L)) = 0$ non servono pi\`u, quindi le ho tolte dalle vecchie (c) e (d).
%\end{AL}

\begin{cor}
\label{riso}
Let $X$ be a normal Cohen-Macaulay projective irreducible threefold, and let $L$ be a very ample line bundle on $X$. Let $\mathcal E, \mathcal F$ be two locally free sheaves on $X$ such that $\rk(\mathcal F) = \rk(\mathcal E)+1, \det \mathcal E \simeq \det \mathcal F$ and $\mathcal E^* \otimes \mathcal F$ is ample and globally generated. Let $\phi : \mathcal E \to \mathcal F$ be a general morphism and let $C = D_{\rk(\mathcal E)-1}(\phi)$ be its degeneracy locus. Suppose that
\begin{itemize}
\item[(a)] $H^i(\O_X) = 0$ for $i > 0$
\item[(b)] $H^1(\mathcal F(L)) = 0$
\item[(c)] $H^2(\mathcal E(L)) = 0$
\item[(d)] $H^0(\mathcal F \otimes \omega_X(L)) = H^1(\mathcal F \otimes \omega_X(L)) = 0$
\item[(e)] $H^1(\mathcal E \otimes \omega_X(L)) = H^2(\mathcal E \otimes \omega_X(L)) = 0$
\item[(f)] $H^2(\mathcal F \otimes \mathcal F^*) = H^3(\mathcal E \otimes \mathcal F^*) = 0$
\item[(g)] $H^1(\mathcal F \otimes \mathcal E^*) = H^2(\mathcal E \otimes \mathcal E^*) = 0$
\item[(h)] there is a very ample line bundle $H$ on $X$ such that $\mathcal F(L - H)$ is globally generated.
\end{itemize}
Then conditions (i)-(v) of Lemma \ref{lemmetto} are satisfied.
\end{cor}
%VERSIONE PRECEDENTE
%\begin{cor}
%\label{riso}
%Let $X$ be a normal Cohen-Macaulay projective irreducible threefold, and let $L$ be a very ample line bundle on $X$. Let $\mathcal E, \mathcal F$ be two locally free sheaves on $X$ such that $\rk(\mathcal F) = \rk(\mathcal E)+1, \det \mathcal E \simeq \det \mathcal F$ and $\mathcal E^* \otimes \mathcal F$ is ample and globally generated. Let $\phi : \mathcal E \to \mathcal F$ be a general morphism and let $C = D_{\rk(\mathcal E)-1}(\phi)$ be its degeneracy locus. Suppose that
%\begin{itemize}
%\item[(a)] $H^i(\O_X) = 0$ for $i > 0$
%\item[(b)] $H^1(L) = 0$
%\item[(c)] $H^1(\mathcal F(L)) = H^2(\mathcal F(L)) = 0$
%\item[(d)] $H^2(\mathcal E(L)) = H^3(\mathcal E(L)) = 0$
%\item[(e)] $H^0(\mathcal F \otimes \omega_X(L)) = H^1(\mathcal F \otimes \omega_X(L)) = 0$
%\item[(f)] $H^1(\mathcal E \otimes \omega_X(L)) = H^2(\mathcal E \otimes \omega_X(L)) = 0$
%\item[(g)] $H^2(\mathcal F \otimes \mathcal F^*) = H^3(\mathcal E \otimes \mathcal F^*) = 0$
%\item[(h)] $H^1(\mathcal F \otimes \mathcal E^*) = H^2(\mathcal E \otimes \mathcal E^*) = 0$
%\item[(i)] there is a very ample line bundle $H$ on $X$ such that $\mathcal F(L - H)$ is globally generated.
%\end{itemize}
%Then conditions (i)-(v) of Lemma \ref{lemmetto} are satisfied.
%\end{cor}
\begin{proof}
First note that (i)  %\todo{Ho tolto "and (ii)"}
%and (ii) 
of Lemma \ref{lemmetto} is (a). 

Moreover, \cite[Ch.\ VI, \S 4, page 257]{acgh} implies that the ideal sheaf of $C$ 
%\todo{ho precisato di chi  il fascio di ideali}  
has a resolution 
\begin{equation}
\label{ris}
0 \to \mathcal E \to \mathcal F \to \I_{C/X} \to 0
\end{equation}
so that (v) of Lemma \ref{lemmetto} follows by (h) of this Corollary. Then we get the exact sequences
\[ 0 \to \mathcal E \otimes \mathcal F^* \to \mathcal F \otimes \mathcal F^* \to \I_{C/X} \otimes \mathcal F^* \to 0 \]
and
\[ 0 \to \mathcal E \otimes \mathcal E^* \to \mathcal F \otimes \mathcal E^* \to \I_{C/X} \otimes \mathcal E^* \to 0. \]
Using (f) and (g) we deduce that $H^2(\I_{C/X} \otimes \mathcal F^*) = H^1(\I_{C/X} \otimes \mathcal E^*) = 0$.
Applying $\Hom_{\O_X}(-,\I_{C/X})$ to \eqref{ris} we get the exact sequence
\begin{equation}
\label{uno}
\Ext^1_{\O_X}(\mathcal E, \I_{C/X}) \to \Ext^2_{\O_X}(\I_{C/X}, \I_{C/X}) \to \Ext^2_{\O_X}(\mathcal F, \I_{C/X}).
\end{equation}
Now $\Ext^1_{\O_X}(\mathcal E, \I_{C/X}) \simeq H^1(\I_{C/X} \otimes \mathcal E^*) = 0$, and $\Ext^2_{\O_X}(\mathcal F, \I_{C/X}) \simeq  H^2(\I_{C/X} \otimes \mathcal F^*) = 0$.
By \eqref{uno} and Lemma \ref{kleppe} it follows that $H^1(N_{C/X}) = 0$, that is (ii) of Lemma \ref{lemmetto}.

From \eqref{ris} we also have the exact sequence
\[ 0 \to \mathcal E(L) \to \mathcal F(L) \to \I_{C/X}(L) \to 0 \]
and, using (b) and (c), we get (iii) of Lemma \ref{lemmetto}. 
%\begin{AL}
%Qui c'era una frase che ho tolto perch\`e riguardava la vecchia (b)
%\end{AL}
%and $H^2(\I_{C/X}(L)) = 0$.
%From
%\[ 0 \to \I_{C/X}(L) \to L \to L_{|C} \to 0 \]
%using (b), we obtain (ii)  %\todo{Ho sostituito "(iii)" con "(ii)"} 
%of Lemma \ref{lemmetto}. 

Finally \eqref{ris} gives an exact sequence
\[ 0 \to \mathcal G \to \mathcal E \otimes \omega_X(L) \to \mathcal F \otimes \omega_X(L) \to \I_{C/X} \otimes \omega_X(L) \to 0 \]
where $\mathcal G$ is a sheaf with support of dimension at most $1$. Using (d) and (e), we get (iv) of Lemma \ref{lemmetto}.
\end{proof}

\section{Proof of main results}

Putting together our tools, Lemma \ref{lemmetto}, Proposition \ref{minori} %\todo{Ho aggiunto "Proposition \ref{minori}"} 
and Corollary \ref{riso}, we now proceed to the proofs.

\subsection{Proof of Theorem \ref{esistenza}}
\begin{proof}
Let $\mathcal E = \O_X(-dH)^{\oplus (d-1)}, \mathcal F = \O_X((1-d)H)^{\oplus d}$ and let $\phi: \mathcal E \to \mathcal F$ be a generic morphism. Note that $H^1(\omega_X(H)) = H^2(\O_X(-H)) = 0$ by \cite[Thm.~7.80 (c)]{ShiffSomm}.
%\cite[Prop.\ 1.1 and Rmk.\ 1.2]{aj}. 
Setting $C = D_{d-2}(\phi)$, it follows by the hypotheses that all conditions (a)-(h) of Corollary \ref{riso} are satisfied. Moreover, by Proposition \ref{minori}, %\todo{Ho aggiunto "by Proposition \ref{minori}"} 
$C$ is smooth irreducible, $C \cap \Sing(X) = \emptyset$ and all conditions (i)-(vi) %\todo{Ho sostituito "(viii)" con "(vii)"} 
of Lemma \ref{lemmetto} are satisfied. We then conclude by Lemma \ref{lemmetto}.
\end{proof}

%\begin{remark}
%\label{fano}
%Using \cite[Prop.\ 1.1 and Rmk.\ 1.2]{aj}, it is easily seen that the hypotheses of Theorem \ref{esistenza} and Corollary \ref{densita'} are satisfied, for example, on a normal, $\QQ$-factorial, irreducible Fano threefold with rational singularities $X$ having a very ample line bundle $H$ such that $H^0(K_X+H) = 0$ and $\omega_X(dH)$ is globally generated. In particular this happens when $-K_X = r H$ with $r \ge 2$ and $d \ge r$.
%\end{remark}

\subsection{Proof of Corollary \ref{densita'}}
\begin{proof}
We just note that, since we are working with irreducible normal surfaces with rational singularities, the proof of \cite[\S 5]{CilibertoHarriaMirandaNLLocus} works verbatim on the open subset $U(L)$ of $|L|$.
\end{proof}
%\todo{Nella dim. del Thm.2 sono tornato a $dH$}
\subsection{Proof of Theorem \ref{esistenza-torica}}
\begin{proof}
Note that ${\PP_{\Sigma}}$ is normal and $\mathbb Q$-factorial, because it is  toric and  simplicial.
Let $\mathcal E = \O_{\PP_{\Sigma}}(-(d+2)H)^{\oplus (d+1)}, \mathcal F = \O_{\PP_{\Sigma}}(-(d+1))H)^{\oplus (d+2)}$ and let $\phi: \mathcal E \to \mathcal F$ be a generic morphism. We set $L = - K_{\PP_{\Sigma}} + d H$ and check the conditions of Corollary \ref{riso}. 
%\begin{AL}
%Ho aggiunto la dimostrazione che $L$ \`e molto ampio. Inoltre ho anticipato la verifca di (h).
%\end{AL}

Note that $-K_{\PP_{\Sigma}}-2H$ is globally generated by \cite[Thm.\ 1.6]{Mavlyutov-semi}, whence $L = -K_{\PP_{\Sigma}}-2H + (d+2)H$ is very ample. Now also $\mathcal F(L-H) \simeq \O_{\PP_{\Sigma}}(-K_{\PP_{\Sigma}}-2H)^{\oplus (d+2)}$ is globally generated, and this gives (h). Using the nefness of $-K_{\PP_{\Sigma}} - 2 H$ we see that conditions (a)-(c), (g) and the first vanishing in (f), follow by Demazure's vanishing theorem \cite[Thm.\ 9.2.3]{CoxLSh}. Also conditions (d) and (e) follow by toric Serre duality \cite[Thm.\ 9.2.10]{CoxLSh} and by Bott-Danilov-Steeenbrink's vanishing theorem \cite[Chapt.\ 3]{Oda88}. Let us see that also the second vanishing in (f) holds, namely that $H^3(\O_{\PP_{\Sigma}}(- H)) = 0$. In fact if $H^3(\O_{\PP_{\Sigma}}(- H)) \neq 0$, then, by toric Serre duality, $H^0(\O_{\PP_{\Sigma}}(K_{\PP_{\Sigma}} + H)) \neq 0$ and therefore also $H^0(\O_{\PP_{\Sigma}}(2K_{\PP_{\Sigma}} + 2H)) \neq 0$. But the latter is dual to $H^3(\O_{\PP_{\Sigma}}(-K_{\PP_{\Sigma}} - 2 H)) = 0$ by Demazure's vanishing theorem, a contradiction.

Therefore all the conditions of Proposition \ref{minori} %\todo{Ho aggiunto "Proposition \ref{minori}"} 
and Corollary \ref{riso} are satisfied and we deduce that conditions (i)-(v) of Lemma \ref{lemmetto} are also satisfied. Since $K_{\PP_{\Sigma}} + L = dH$ is globally generated we also have (vi) of Lemma \ref{lemmetto}. We then conclude by Lemma \ref{lemmetto}.
%\begin{AL}(NUOVA)
%Ho percentualizzato la frase che c'era sotto
%\end{AL}
%Corollary 4.13 and Proposition 3.5 in \cite{bg2}  imply the lower bound estimate on the codimension.
%VERSIONE VECCHIA
%Therefore all the conditions of Proposition \ref{minori} %\todo{Ho aggiunto "Proposition \ref{minori}"} 
%and Corollary \ref{riso} are satisfied and we deduce that conditions (i)-(v) of Lemma \ref{lemmetto} are also satisfied.
%
%Moreover, note that $K_{\PP_{\Sigma}} + L = dH$ is globally generated and so is $\mathcal F(L-H) \simeq \O_{\PP_{\Sigma}}(-K_{\PP_{\Sigma}}-2H)^{\oplus d+2}$ by \cite[Thm.\ 1.6]{Mavlyutov-semi}, so that also (h) of Corollary \ref{riso} whence also (vi) of Lemma \ref{lemmetto} is satisfied. We then conclude by Lemma \ref{lemmetto}.
\end{proof}
%\begin{AG}
%Ho aggiunto l'ultima riga.
%\end{AG}
%\begin{AL}(NUOVA)
%OK
%\end{AL}

\subsection{Proof of Corollary \ref{densita'-torica}}
\begin{proof}
The first part of the statement is proved as in  Corollary \ref{densita'}.  Note that $H$ is 0-regular, see Section \ref{esempi}.
Corollary 4.13 and Proposition 3.6 in \cite{bg2}  then  imply the lower bound estimate on the codimension.
%\begin{AL}(NUOVA)
%Ho aggiunto una riga
%\end{AL}
The upper bound follows by Proposition \ref{codmax}.
\end{proof}

\bigskip
\frenchspacing
\bibliographystyle{siam}
\bibliography{paper}

\end{document}